\documentclass[a4paper]{amsart}
\usepackage{latexsym}\usepackage{ifthen}\usepackage[leqno]{amsmath}
\usepackage{enumerate}\usepackage{calc}\usepackage{hyphenat}
\usepackage{amstext,amsbsy,amsopn,amsthm,amsgen}
\usepackage{amsfonts,amscd,amsxtra,upref}
\usepackage{mathrsfs}\usepackage{euscript}\usepackage{amssymb}
\usepackage{hyperref}
\usepackage{color}
\swapnumbers

\theoremstyle{plain}
\makeatletter\@namedef{subjclassname@2020}{\textup{2020} Mathematics Subject Classification}
\makeatother
\newtheorem{Thm}{Theorem}[section]
\newtheorem{Lem}[Thm]{Lemma}
\newtheorem{Cor}[Thm]{Corollary}
\newtheorem{Pro}[Thm]{Proposition}
\theoremstyle{definition}
\newtheorem{Def}[Thm]{Definition}
\newtheorem{Exm}[Thm]{Example}
\theoremstyle{remark}
\newtheorem{Rem}[Thm]{Remark}
\numberwithin{equation}{section}

\newcommand{\ITE}[3]{\ifthenelse{#1}{#2}{#3}}\newcommand{\ITEE}[4][]{\ITE{\equal{#2}{#3}}{#4}{#1}}
 \ITE{\isundefined{\texorpdfstring}}{\newcommand{\texorpdfstring}[2]{#1}}{}

\newcommand{\myData}[1][]{
 \author[P.\ Niemiec]{Piotr Niemiec}
 \address{\ITEE{#1}{*}{P.\ Niemiec{}\\}
  Wydzia\l{} Matematyki i~Informatyki\\Uniwersytet Jagiello\'{n}ski\\
  ul.\ \L{}o\-ja\-sie\-wi\-cza 6\\30-348 Krak\'{o}w\\Poland}
 \email{piotr.niemiec@uj.edu.pl}
 }

\newenvironment{cor}[2][]{\ITEE[{\begin{Cor}[#1]}]{#1}{}{\begin{Cor}}\label{cor:#2}}{\end{Cor}}
\newenvironment{dfn}[2][]{\ITEE[{\begin{Def}[#1]}]{#1}{}{\begin{Def}}\label{def:#2}}{\end{Def}}
\newenvironment{exm}[2][]{\ITEE[{\begin{Exm}[#1]}]{#1}{}{\begin{Exm}}\label{exm:#2}}{\end{Exm}}
\newenvironment{lem}[2][]{\ITEE[{\begin{Lem}[#1]}]{#1}{}{\begin{Lem}}\label{lem:#2}}{\end{Lem}}
\newenvironment{pro}[2][]{\ITEE[{\begin{Pro}[#1]}]{#1}{}{\begin{Pro}}\label{pro:#2}}{\end{Pro}}
\newenvironment{rem}[2][]{\ITEE[{\begin{Rem}[#1]}]{#1}{}{\begin{Rem}}\label{rem:#2}}{\end{Rem}}
\newenvironment{thm}[2][]{\ITEE[{\begin{Thm}[#1]}]{#1}{}{\begin{Thm}}\label{thm:#2}}{\end{Thm}}

\newcommand{\COR}[2][!]{\ITEE{#1}{!}{Corollary~}\ITEE{#1}{s}{Corollaries~}\textup{\ref{cor:#2}}}
\newcommand{\LEM}[2][!]{\ITEE{#1}{!}{Lemma~}\ITEE{#1}{s}{Lemmas~}\textup{\ref{lem:#2}}}
\newcommand{\THM}[2][!]{\ITEE{#1}{!}{Theorem~}\ITEE{#1}{s}{Theorems~}\textup{\ref{thm:#2}}}

\newcommand{\FFF}{\mathbb{F}}
\newcommand{\KKK}{\mathbb{K}}
\newcommand{\NNN}{\mathbb{N}}
\newcommand{\RRR}{\mathbb{R}}
\newcommand{\ZZZ}{\mathbb{Z}}
\newcommand{\LlL}{\EuScript{L}}

\newcommand{\dd}{\colon}
\newcommand{\df}{\stackrel{\textup{def}}{=}}
\newcommand{\epsi}{\varepsilon}
\newcommand{\scalar}[2]{\left\langle#1,#2\right\rangle}
\newcommand{\scalarr}{\langle\cdot,\mathrm{-}\rangle}
\newcommand{\varempty}{\varnothing}

\newcommand{\diam}{\operatorname{diam}}
\newcommand{\UP}[1]{\textmd{\textup{#1}}}

\begin{document}

\title{Two-point dilation-homogeneous metric spaces}
\myData\thanks{Research supported by the National Center of Science, Poland
 under the Weave-UNISONO call in the Weave programme [grant no
 2021/03/Y/ST1/00072].}
\begin{abstract}
The main aim of the paper is to give a full classification (up to isometry) of
all metric spaces \(X\) with the following two properties: \(X\) contains
a compact set with non-empty interior; and for any three distinct points \(a,
b\) and \(c\) of \(X\) there exists a (bijective) dilation on \(X\) that fixes
\(a\) and sends \(b\) to \(c\). As a consequence, we obtain a new
characterisation of the Euclidean spaces \(\RRR^n\) (with finite \(n > 0\)):
these are (up to isometry) precisely all metric spaces that have the above two
properties, and (in addition) contain three distinct points \(x, y, z\) that
are metrically collinear (that is, for which \(d(x,z) = d(x,y)+d(y,z)\)).
\end{abstract}
\maketitle

\section{Introduction}

In the 50's of the 20th century Wang \cite{wan} and Tits \cite{tit} gave a full
characterisation (up to isometry) of all connected locally compact metric spaces
that are two-point homogeneous (that is, such spaces that for any two equidistal
pairs of points there is a bijective isometry that sends one of these pairs onto
the other). In the same paper the second of these authors also classified all
(the so-called) two-transitive Lie groups. A modern approach to this last issue
was presented by Kramer in \cite{kra}. All these results deal with locally
compact spaces that are not totally disconnected. Although for connected
(separable locally compact metrisable) spaces a nice classification is possible,
a similar result for totally disconnected spaces seems to be unreachable. So,
a natural question arises whether there is a similar class (to two-point
homogeneous metric spaces) which makes it possible to give a full classification
(up to isometry) of \emph{all} locally compact metric spaces from that class,
and simultaneously that contains all Euclidean spaces. The main aim of this
paper is to introduce such a class. It is formed by all locally compact spaces
on which the transformation group of all dilations acts two-transitively. It
follows from our classification that all Euclidean spaces can nicely be
characterised (see \THM{R-n} below); and that each space from our class, apart
from the so-called degenerate spaces, has the property of extending arbitrary
partial dilations to global ones (consult \COR{abs-homo} below).\par
To formulate the main result of the paper, let us introduce necessary notation.
By \(d_e\) we will denote the Euclidean metric on \(\RRR^n\) (that is,
\(d_e(x,y) = \|x-y\|_2\) where \(\|x\|_2 = \sqrt{\sum_{k=1}^n x_k^2}\) for \(x =
(x_1,\ldots,x_n)\)). Further, for each prime power \(q = p^n\) (where \(p\) is
a prime and \(n > 0\) is an integer) let \(\FFF_q\) denote the (unique) field of
size \(q\) and \(\FFF_q((X))\) stand for the field of formal power series over
\(\FFF_q\). That is, \(\FFF_q((X))\) consists of all power series of the form
\(\sum_{n=k}^{\infty} a_n X^n\) where \(k \in \ZZZ\) (is arbitrary) and \(a_n
\in \FFF_q\) for all \(n \geq k\). For each non-zero \(v = \sum_{n=k}^{\infty}
a_n X^n\) from \(\FFF_q((X))\) we denote by \(|v|\) the least integer \(n\) such
that \(a_n \neq 0\); and, in addition, \(|0| \df \infty\).\par
Finally, let \(\delta\) stand for the discrete metric (on an arbitrary set
\(Y\)); that is, \(\delta(a,b) = 1\) for all distinct \(a\) and \(b\). We call
a metric space \emph{Heine-Borel} if all its closed balls are compact. Under
the convention that \(b^{-\infty} \df 0\) for each real \(b > 1\), the main
result of the paper reads as follows (fot the definition of two-point
dilation-homogeneity, see the next subsection).

\begin{thm}{main}
Each two-point dilation-homogeneous metric space that has more than one point
and contains a compact set with non-empty interior is isometric to precisely one
of the spaces listed below.
\begin{enumerate}[\upshape(Type1)]\addtocounter{enumi}{-1}
\item \(D_{\alpha,r}\) (where \(\alpha > 1\) and \(r > 0\) are, respectively,
 a cardinal number and a real number) consisting of a fixed set of cardinality
 \(\alpha\) equipped with the metric \(r \delta\).
\item \(F_{n,a,b}\) (where \(n > 1\) is an integer and \(1 \leq a < b\) are two
 real numbers) consisting of the set \(\prod_{s=1}^N \FFF_{p_s^{k_s}}\) equipped
 with the metric \(\rho((x_s)_{s=1}^N,(y_s)_{s=1}^N) = a b^{-\min_s |x_s-y_s|}\)
 where \(n = \prod_{s=1}^N p_s^{k_s}\) is the prime decomposition of \(n\).
\item \(R_{n,\alpha}\) (where \(n > 0\) is an integer and \(0 < \alpha \leq 1\)
 is a real number) consisting of the set \(\RRR^n\) equipped with the metric
 \(d_e^{\alpha}\).
\end{enumerate}
In particular, each two-point dilation-homogeneous metric space that contains
a compact set with non-empty interior and whose metric attains at least two
positive values is Heine-Borel.
\end{thm}

As a consequence, we obtain the following two results.

\begin{thm}{R-n}
A metric space is isometric to \((\RRR^n,d_e)\) for some \(n > 0\) iff it is
two-point dilation-homogeneous and contains a compact set with non-empty
interior as well as three distinct points that are metrically collinear.
\end{thm}

\begin{cor}{abs-homo}
If a two-point dilation-homogeneous metric space is Heine-Borel, then each
dilation between its arbitrary two subsets extends to a dilation of a global
space.
\end{cor}

Another consequence of \THM{main} that is worth noting here is that each
two-point dilation-homogeneous locally compact metric space admits a structure
of an Abelian group with respect to which its metric is invariant.\par
The paper is organised as follows. Below we fix notation and terminology. In
the next section we collect basic properties of spaces under the question as
well as of dilation groups of metric spaces. We also introduce there three types
of investigated spaces and give a full classification of those of them that are
\emph{degenerate} (that is, of type 0). The third part is devoted to the full
classification of type 1 spaces. The key property proved therein is that all
such spaces are ultrametric. The fourth, final section deals with type 2 spaces.
A full classification of them is a direct consequence of Tits' classification of
two-transitive Lie groups. (However, we follow Kramer's paper \cite{kra} there.)

\subsection*{Notation and terminology}
A \emph{dilation} on a metric space \((X,d)\) is a bijection \(u\dd X \to X\)
such that
\(d(u(x),u(y)) = r u(x,y)\)
for all \(x, y \in X\) and some positive real constant \(r\). The above \(u\) is
an \emph{isometry} if the above \(r\) is equal to \(1\). We call \((X,d)\)
\emph{two-point dilation-homogeneous} if for any three distinct points \(a, b\)
and \(c\) of \(X\) there exists a dilation \(u\dd X \to X\) such that \(u(a) =
a\) and \(u(b) = c\). The group of all dilations on \(X\) will be denoted by
\(\Lambda(X)\) or, more specifically, \(\Lambda(X,d)\) (where \(d\) is
the metric of the space \(X\)). Similarly, \(X\) is said to be \emph{metrically
homogeneous} if for any two points \(a\) and \(b\) of \(X\) there exists
an isometry that sends \(a\) onto \(b\).\par
Three points \(a, b\) and \(c\) of \(X\) are called \emph{metrically collinear}
if \(d(x,z) = d(x,y)+d(y,z)\). The metric \(d\) is an \emph{ultrametric} if
\(d(x,z) \leq \max(d(x,y),d(y,z))\) for any \(x, y, z \in X\). A metric \(p\) on
an Abelian group \((G,+)\) is \emph{invariant} if \(p(x+z,y+z) = p(x,y)\) for
all \(x, y, z \in G\).\par
Closed balls in \(X\) will be denoted by \(\bar{B}_X(a,r)\) or, simply,
\(\bar{B}(a,r)\). All the notations used in \THM{main} remain valid throughout
the rest of the paper.\par
\textbf{From now on to the end of the paper, \boldmath{\((X,d)\)} is a two-point
dilation-homogeneous metric space that has more than one point and contains
a compact set with non-empty interior.}

\section{Preliminaries}

The proof of the above simple result is left to the reader.

\begin{lem}{homo-R}
For each \(u \in \Lambda(X)\) there exists a \textbf{unique} positive real
number \(c = |u|\) such that \(d(u(x),u(y)) = c d(x,y)\) for all \(x, y \in X\).
The function \((\Lambda(X),\circ) \ni u \mapsto |u| \in ((0,\infty),\cdot)\) is
a group homomorphism.
\end{lem}

\begin{dfn}{type}
Set \(\Gamma(X) \df \{d(x,y)\dd\ x \neq y\}\). \(X\) is said to be
\begin{itemize}
\item \emph{type 0} or \emph{degenerate} if \(\Gamma(X)\) consists of a single
 number; otherwise it is called \emph{non-degenerate};
\item \emph{type 1} if it is non-degenerate and \(\Gamma(X)\) is not dense in
 \((0,\infty)\);
\item \emph{type 2} if \(\Gamma(X)\) is dense in \((0,\infty)\).
\end{itemize}
(Note that \(\Gamma(X) \neq \varempty\).)
\end{dfn}

\begin{lem}{two-point}
For any two pairs \((a,b)\) and \((p,q)\) such that \(a \neq b\) and \(p \neq
q\) there is \(u \in \Lambda(X)\) satisfying \(u(a) = p\) and \(u(b) = q\). In
particular, \(X\) is metrically homogeneous.
\end{lem}
\begin{proof}
We may and do assume that \(X\) has more than two points. For simplicity, we
will write (here in this proof) \((a,b) \equiv (p,q)\) to express that
the conclusion of the lemma holds. First note that \((x,y) \equiv (x,z)\) and
\((y,x) \equiv (z,x)\) for any \(x \in X\) and two points \(y\) and \(z\) of
\(X\) each of which is distinct from \(x\) (thanks to two-point
dilation-homogeneity of \(X\)). Now choose arbitrary point \(c \in X\) that
differs from both \(a\) and \(b\). Then:
\begin{itemize}
\item if \(a = p\), then \((a,b) = (p,b) \equiv (p,q)\);
\item if \(a \neq p\), then \((a,b) \equiv (c,b) \equiv (c,a) \equiv (p,a)
 \equiv (p,q)\),
\end{itemize}
which proves the first part of the lemma. Consequently, \((a,b) \equiv (b,a)\)
for any two distinct points \(a\) and \(b\) of \(X\), which implies that \(X\)
is metrically homogeneous.
\end{proof}

\begin{lem}{group}
Let \(c\) be an arbitrary element of \(\Gamma(X)\) and let \(V \df \{|u|\dd\
u \in \Lambda(X)\}\). Then \(V\) is a subgroup of \(((0,\infty),\cdot)\) and
\(\Gamma(X) = c \cdot V\ (= \{cv\dd\ v \in V\})\). In particular, if \(X\) is
type 1, then \(\Gamma(X) = \{a b^k\dd\ k \in \ZZZ\}\) for a unique pair
\((a,b)\) of reals such that \(1 \leq a < b\).
\end{lem}
\begin{proof}
There are two distinct points \(x\) and \(y\) of \(X\) such that \(d(x,y) = c\).
We infer from \LEM{two-point} that \(\Gamma(X) = \{c d(u(x),u(y))\dd\ u \in
\Lambda(X)\}\) and, consequently, that \(\Gamma(X) = c V\). The additional claim
now easily follows.
\end{proof}

\begin{lem}{H-B}
If \(X\) is non-degenerate, then it is Heine-Borel and non-compact. In
particular, if \(X\) is type 2, then \(\Gamma(X) = (0,\infty)\).
\end{lem}
\begin{proof}
Fix \(a \in X\) and a radius \(r > 0\). Since \(X\) contains a compact set with
non-empty interior, we infer that there is \(b \in X\) and \(\epsi > 0\) such
that \(\bar{B}(b,\epsi)\) is compact. E.g. thanks to Lemmas~\ref{lem:two-point}
and \ref{lem:group}, there exists \(u \in \Lambda(X)\) such that \(u(b) = a\)
and \(|u| > \frac{r}{\epsi}\). Then \(\bar{B}(a,r) \subset u(\bar{B}(b,\epsi))\)
and therefore \(\bar{B}(a,r)\) is compact. Since the set \(\Gamma(X)\) is
unbounded, \(X\) is non-compact.\par
Now assume \(X\) is type 2. Since \(X\) is metrically homogeneous (by
\LEM{two-point}), it follows that for each \(r > 1\) the set \(\{d(a,x)\dd\
\frac1r \leq d(a,x) \leq r\}\) is compact (where \(a\) is arbitrarily fixed
element of \(X\)) and coincides with \(\Gamma(X) \cap [\frac1r,r]\). So,
\([\frac1r,r] \subset \Gamma(X)\), as \(\Gamma(X)\) is dense in \((0,\infty)\).
\end{proof}

The following is quite an easy result. We skip its trivial proof.

\begin{pro}{type-0}
The class of all degenerate two-point dilation-homogeneous metric spaces
coincides with the class of all spaces whose metrics are of the form \(r
\delta\) where \(r > 0\). In particular, all such spaces are locally compact and
Heine-Borel are precisely those of them whose underying spaces are finite.
Consequently, each such a space is uniquely determined (up to isometry) by its
size and the diameter; and the conclusion of \COR{abs-homo} is valid for type 0
spaces.
\end{pro}

The last result of this section generalises a well-known property of isometry
groups of Heine-Borel spaces. Readers interested in isometry groups of metric
spaces are referred to \cite{g-k} or \cite{pn1,pn2}.

\begin{lem}{loc-comp}
For any Heine-Borel metric space \((Z,\rho)\) the group \(G \df
\Lambda(Z,\rho)\) equipped with the pointwise convergence topology is
a topological group that is both locally compact and \(\sigma\)-compact. More
specifically, for any two distinct points \(a\) and \(b\) in \(Z\) and
an arbitrary real number \(R\) greater than \(1\) the set
\[U(a,b,R) \df \{u \in G\dd\ \max(\rho(u(a),a),\rho(u^{-1}(a),a),\rho(u(b),b),
\rho(u^{-1}(b),b)) \leq R\}\]
is a compact neighbourhood of the identity of \(G\).
\end{lem}
\begin{proof}
Without loss of generality, we may and do assume that \(Z\) has more than one
point. So, we may (and do) fix two distinct points \(a\) and \(b\) of \(Z\). For
the purpose of this proof, let us call a function \(v\dd Z \to Z\)
\emph{homothetic} if \(\rho(v(x),v(y)) = c \rho(x,y)\) for all \(x, y \in Z\)
and some positive real constant \(c\) (so, a homothetic function need not to be
surjective). As for dilations, we denote this unique constant \(c\) by \(|v|\).
Note that \(G \ni u \mapsto |u| \in (0,\infty)\) is a continuous group
homomorphism, as \(|u| = \frac{\rho(u(a),u(b))}{\rho(a,b)}\). More generally,
if a net \((v_{\sigma})_{\sigma\in\Sigma}\) of homothetic functions on \(Z\)
converges pointwise to a function \(w\dd Z \to Z\), then \(w\) is homothetic and
\(|w| = \lim_{\sigma\in\Sigma} |v_{\sigma}|\). In particular, if, in addition,
also \((v'_{\sigma})_{\sigma\in\Sigma}\) is a net of homothetic functions that
converge pointwise to \(w'\dd Z \to Z\), then the net \((v'_{\sigma} \circ
v_{\sigma})_{\sigma\in\Sigma}\) converges pointwise to \(w' \circ w\). (Indeed,
\(\max(|v_{\sigma}|,|v'_{\sigma}|) \leq M\) for all \(\sigma \geq \sigma_0\) and
some \(\sigma_0\in\Sigma\) and \(M > 0\); then, for all \(\sigma \geq \sigma_0\)
and each \(z \in Z\), \(\rho((v'_{\sigma} \circ v_{\sigma})(z),(w' \circ w)(z))
\leq \rho(v'_{\sigma}(v_{\sigma}(z)),v'_{\sigma}(w(z))+\rho(v'_{\sigma}(w(z)),
w'(w(z))) \leq M \rho(v_{\sigma}(z),w(z))+\rho(v'_{\sigma}(w(z)),w'(w(z))) \to
0\).) We conclude that the binary action of \(G\) is continuous, and that \(G\)
is metrisable (as \(Z\) is separable and for any countable dense subset \(D\) of
\(Z\) the assignment \(G \ni g \mapsto g\restriction{D} \in Z^D\) defines
a topological embedding of \(G\) into a metrisable space \(Z^D\) of all
functions from \(D\) to \(Z\), equipped with the pointwise convergence
topology). Moreover, if \(u_1,u_2,\ldots \in G\) converge pointwise to \(g \in
G\), then
\[\rho(u_n^{-1}(z),g^{-1}(z)) = \frac{\rho(g(g^{-1}(z)),u_n(g^{-1}(z)))}{|u_n|}
\to 0 \quad (n\to\infty)\]
(because \(\lim_{n\to\infty} |u_n| = |g| > 0\)). So, \(G\) is a topological
group.\par
Now we will show that \(U(a,b,R)\) is a compact set in \(G\) (it is clear that
it is a neighbourhood of the identity). To this end, fix for a while arbitrary
\(f \in U(a,b,R)\) and note that \(f^{-1} \in U(a,b,R)\) as well. Moreover,
\(|f| \cdot \rho(a,b) = \rho(f(a),f(b)) \leq \rho(f(a),a)+\rho(a,b)+\rho(b,f(b))
\leq \rho(a,b)+2R\). So, \(|f| \leq C\) where \(C \df 1+\frac{2R}{\rho(a,b)}\).
In particular, for any \(z \in Z\), \(\rho(f(z),a) \leq \rho(f(z),f(a))+
\rho(f(a),a)) \leq C \rho(z,a)+R\). Equivalently,
\begin{equation}\label{eqn:aux1}
f(z) \in K_z \qquad (z \in Z)
\end{equation}
where \(K_z \df \bar{B}(a,C \rho(z,a)+R)\) is a compact subset of \(Z\). Now
take an arbitrary sequence \((u_n)_{n=1}^{\infty}\) of functions from
\(U(a,b,R)\). Then also \(u_n^{-1} \in U(a,b,R)\). If \(D\) is a countable dense
set in \(Z\), then \((u_n\restriction{D},u_n^{-1}\restriction{D}) \in
(\prod_{z \in D} K_z)^2\) and it follows from the compactness of the last
Cartesian product that after passing to a subsequence, we may and do assume that
\(\lim_{n\to\infty} u_n(z) = v_o(z)\) and \(\lim_{n\to\infty} u_n^{-1}(z) =
w_o(z)\) for all \(z \in D\) and some functions \(v_o, w_o\dd D \to Z\). Now
a standard argument shows that both \(v_o\) and \(w_o\) extend to \(v, w\dd Z
\to Z\) and that the maps \(u_1,u_2,\ldots\) (resp. \(u_1^{-1},u_2^{-1},
\ldots\)) converge pointwise to \(v\) and \(w\). In particular, both \(v\) and
\(w\) are homothetic and the maps \(u_n \circ u_n^{-1}\) (resp. \(u_n^{-1} \circ
u_n\)) converge pointwise to \(v \circ w\) (resp. to \(w \circ v\)).
Consequently, \(v, w \in G\) and \(w = u^{-1}\). Hence \(v \in U(a,b,R)\) and
thus \(G\) is locally compact. A note that \(G = \bigcup_{n=1}^{\infty}
U(a,b,n)\) completes the proof.
\end{proof}

\section{Type 1 spaces}

This section is devoted to a full classification (up to isometry) of all
(two-point dilation-homogeneous locally compact) type 1 spaces. To this end, we
fix such a space \((X,d)\). It follows from \LEM{group} that there are unique
real numbers \(a\) and \(b\) such that \(1 \leq a < b\) and
\begin{equation}\label{eqn:Gamma(X)-1}
\Gamma(X) = \{ab^k\dd\ k \in \ZZZ\}.
\end{equation}
From now on to the end of this section, \(a\) and \(b\) are reserved to denote
these two numbers.\par
The following result is crucial for classifying type 1 spaces.

\begin{lem}{ultram}
The metric of \(X\) is an ultrametric.
\end{lem}
\begin{proof}
Assume, on the contrary, that \(d\) is not an ultrametric. This means that there
are three distinct points \(x, y, z \in X\) such that \(d(x,y) \leq d(y,z) <
d(x,z)\). Because of \eqref{eqn:Gamma(X)-1} and since for each \(k \in \ZZZ\),
\(\Lambda(X)\) contains a dilation \(\phi_k\) such that \(|\phi_k| = b^k\), we
may and do assume that \(d(x,y) = a\). Express both \(d(y,z)\) and \(d(x,z)\) in
the form \(d(y,z) = ab^k\) and \(d(x,z) = ab^n\) where \(k, n \in \ZZZ\). Then
\(0 \leq k < n\) and (by the triangle inequality) \(b^n \leq 1+b^k \leq
1+b^{n-1}\), which implies that \(b^{n-1} \leq \frac{1}{b-1}\). So, among all
triples \((x,y,z) \in X^3\) such that \(a = d(x,y) \leq d(y,z) < d(x,z)\) there
is one such that \(d(x,z)\) is the largest possible. We fix such a `maximal'
triple \((x,y,z)\) and denote \(d(y,z) = ab^k\) and \(d(x,z) = ab^N\) (then \(0
\leq k < N\)). We conclude from \LEM{two-point} that there exists \(\psi \in
\Lambda(X)\) for which \(\psi(x) = x\) and \(\psi(y) = z\). Set \(w \df
\psi(z)\) and observe that \(|\psi| = b^N\) and therefore
\[d(x,w) = b^N d(x,z) = ab^{2N}, \qquad d(z,w) = b^N d(y,z) = ab^{N+k}.\]
Now denote \(d(y,w) = ab^p\) (where \(p \in \ZZZ\)) and notice that:
\begin{itemize}
\item if \(p > 2N\), then \(a = d(x,y) < ab^{2N} = d(x,w) < ab^p = d(y,w)\) and
 \(p > N\);
\item if \(p < 0\), then \(a = d(\phi_{-p}(y),\phi_{-p}(w)) < ab^{-p} =
 d(\phi_{-p}(y),\phi_{-p}(x)) < ab^{2N-p} = d(\phi_{-p}(w),\phi_{-p}(x))\) and
 \(2N-p > N\);
\item if \(0 \leq p < 2N\), then \(a = d(x,y) \leq d(y,w) < ab^{2N} = d(x,w)\)
 and \(2N > N\),
\end{itemize}
and each of the above cases contradicts the `maximality' of \(N\). So, we
conclude that \(p = 2N\). But then \(a = d(\phi_{-k}(y),\phi_{-k}(z)) < ab^N =
d(\phi_{-k}(z),\phi_{-k}(w)) < ab^{2N-k} = d(\phi_{-k}(y),\phi_{-k}(w))\) and
\(2N-k > N\), which again leads to a contradiction, and completes the proof.
\end{proof}

The following is a kind of folklore (see, e.g., \cite[Theorem~4.2]{m-s}).

\begin{thm}{ultram-ultrah}
Every metrically homogeneous ultremetric space is ultrahomogeneous; that is,
isometries between its arbitrary finite subsets extend to global isometries.
\end{thm}

The following is a direct consequence of the above two results and of
\LEM{two-point}.

\begin{cor}{ultra}
\(X\) is ultrahomogeneous.
\end{cor}

For the purposes of the next results, we introduce the following auxiliary

\begin{dfn}{charact-1}
Let \((Z,\rho)\) be a (two-point dilation-homogeneous locally compact) space of
type 1, and let \(\Gamma(Z) = \{pq^k\dd\ k \in \ZZZ\}\) where \(1 \leq p < q\).
Further, let \(N\) denote the supremum of all sizes of the sets \(A \subset Z\)
such that \(\rho(x,y) = p\) for all \(x, y \in A\). Then \(N \in \NNN \setminus
\{0,1\}\), by \LEM{H-B}. We use \((N(Z),a(Z),b(Z))\) to denote the above three
numbers \((N,p,q)\).
\end{dfn}

\begin{thm}{isometric-1}
A space \(Z\) of type 1 is isometric to \(X\) iff
\[(N(Z),a(Z),b(Z)) = (N(X),a(X),b(X)).\]
\end{thm}
\begin{proof}
The `only if' part is clear. To show the `if' part, we set \(N \df N(X) > 1\)
(recall that \(a = a(X)\) and \(b = b(X)\)), and denote by \(\Omega(N)\) the set
consisting of an element \(\varempty\) and of all sequences
\((\mu_n)_{n=0}^{\infty} \subset \ZZZ\) such that \(0 < \mu_1 < N\) and \(0 <
\mu_n \leq N\) for all \(n > 1\). Finally, for two distinct sequences \(\mu =
(\mu_n)_{n=0}^{\infty}\) and \(\nu = (\nu_n)_{n=0}^{\infty}\) from \(\Omega(N)\)
let \(\eta(\mu,\nu)\) stand for the least index \(k \geq 0\) such that \(\mu_k
\neq \nu_k\), and let
\[D(\mu,\nu) = D(\nu,\mu) \df \begin{cases}
0 & \UP{if } \mu = \nu\\
\mu_0 & \UP{if } \mu \neq \varempty = \nu\\
ab^{\max(\mu_0,\nu_0)} & \UP{if } \eta(\mu,\nu) = 0\\
ab^{\mu_0-\eta(\mu,\nu)+1} & \UP{otherwise}
\end{cases}\]
(\(D\) defines an ultrametric on \(\Omega(N)\), but we do not need this
property). We will show that \(X\) admits a bijection \(\Phi\dd X \to
\Omega(N)\) such that
\begin{equation}\label{eqn:aux44}
d(x,y) = D(\Phi(x),\Phi(y)) \qquad (x, y \in X)
\end{equation}
(which will finish the proof, by applying the same property to \(Z\)). To this
end, we fix an element \(\theta \in X\) and for each integer \(k \in \ZZZ\) we
fix a dilation \(\phi_k \in \Lambda(X)\) satisfying \(\phi_k(\theta) = \theta\)
and \(|\phi_k| = b^k\). Further, for each \(x \in X\) and \(\ell \in \ZZZ\) we
will use \(B(x,\ell)\) (resp. \(S(x,\ell)\)) to denote the set of all \(z \in
X\) such that \(d(x,z) \leq ab^{\ell}\) (resp. \(d(x,z) = ab^{\ell}\)). Further,
we fix a set \(A \subset S(\theta,0)\) for which \(d(x,y) = a\) for any \(x, y
\in A\) and whose size is the largest possible. It follows from both
ultrahomogeneity (cf. \COR{ultra}) and two-point dilation-homogeneity of \(X\),
and from the definition of \(N(X)\) that:
\begin{itemize}
\item \(A\) has \(N-1\) elements (as \(A \cup \{\theta\}\) is a maximal set of
 equidistal points);
\item for each \(k \in \ZZZ\), \(E(\theta,k) \df \phi_k(A)\) is a maximal subset
 of \(S(\theta,k)\) consisting of equidistal points whose distance is \(ab^k\).
\end{itemize}
Similarly, for any \(x \in S(\theta,\ell)\) (where \(\ell \in \ZZZ\)) and each
integer \(k < \ell\) we fix a set \(E(x,k) \subset B(x,k)\) of the largest
possible size such that \(d(u,v) = ab^k\) for all distinct \(u,v \in E(x,k)\).
As before, \(E(x,k)\) has exactly \(N\) elements (by two-point
dilation-homogeneity of \(X\) and since \(A \cup \{\theta\}\) has \(N\)
elements). Finally, express each of the above sets \(E(x,k)\) in the form
\(E(x,k) = \{b_{x,k}^1,\ldots,b_{x,k}^M\}\) where \(M = N\) if \(x \neq \theta\)
and \(M = N-1\) for \(x = \theta\). Since \(d\) is an ultrametric and
\(\Gamma(X)\) has a specific form, we conclude that for any \(k,\ell \in \ZZZ\):
\begin{equation}\label{eqn:aux45}\begin{cases}
S(\theta,\ell) = \bigsqcup_{q=1}^{N-1} B(b_{\theta,\ell}^q,\ell-1) &\\
B(x,k) = \bigsqcup_{q=1}^N B(b_{x,k}^q,k-1)& x \in S(\theta,\ell),\ k < \ell
\end{cases}\end{equation}
(where `\(\bigsqcup\)' means that the union consists of pairwise disjoint sets).
For any finite tuple \((\mu_0,\ldots,\mu_n) \in \ZZZ^{n+1}\) (where \(n \geq
0\)) satisfying \(0 < \mu_1 < N\) (provided that \(n > 0\)) and \(0 < \mu_k \leq
N\) whenever \(1 < k \leq N\), we define \(c(\mu_0,\ldots,\mu_n)\) by
a recursive rule: \(c(\mu_0) \df \theta\) and \(c(\mu_0,\ldots,\mu_{k+1}) \df
b_{c(\mu_0,\ldots,\mu_k),\mu_0-k}^{\mu_{k+1}}\).\par
Now we define \(\Phi\dd X \to \Omega(N)\) as follows: \(\Phi(\theta) \df
\varempty\) and for \(x \neq \theta\), \(\Phi(x)\) is a sequence
\((\mu_n(x))_{n=0}^{\infty}\) defined inductively as follows:
\begin{itemize}
\item \(\mu_0(x)\) is a unique integer such that \(x \in S(\theta,\mu_0(x))\);
\item \(\mu_1(x)\) is a unique integer from \(\{1,\ldots,N-1\}\) such that
 \[x \in B(c(\mu_0(x),\mu_1(x)),\mu_0(x)-1)\]
 (cf. \eqref{eqn:aux45});
\item for \(n > 1\), \(\mu_n(x)\) is a unique integer from \(\{1,\ldots,N\}\)
 such that
 \[x \in B(c(\mu_0(x),\ldots,\mu_n(x)),\mu_0(x)-n)\]
 (again, see \eqref{eqn:aux45}).
\end{itemize}
It follows from the above construction that \(x \in B(c(\mu_0(x),\ldots,
\mu_n(x)),\mu_0(x)-n)\) for any \(n \geq 0\). In particular, if \(\Phi(x) =
\Phi(y)\), then either \(x = y = \theta\) or both \(x\) and \(y\) are different
from \(\theta\) and \(\mu_n(x) = \mu_n(y)\) for all \(n\)---then \(x,y \in
B(c(\mu_0(x),\ldots,\mu_n(x)),\mu_0(x)-n)\) and thus \(d(x,y) \leq 2a
b^{\mu_0(x)-n} \to 0\ (n\to\infty)\), which yields that \(x = y\). This shows
that \(\Phi\) is one-to-one. On the other hand, if \((\mu_n)_{n=0}^{\infty}\) is
an arbitrary sequence from \(\Omega(N)\), then \(S(\theta,\mu_0)\) and all
the sets \(B(c(\mu_0,\ldots,\mu_n),\mu_0-n)\) (where \(n > 0\)) form
a decreasing sequence of compact non-empty subsets of \(X\) (thanks to
\LEM{H-B} and \eqref{eqn:aux45}). So, their intersection is non-empty, say it
contains \(x\). Then \(x \neq \theta\) and \(\Phi(x) = (\mu_n)_{n=0}^{\infty}\),
and \(\Phi\) is a bijection. It remains to check that \eqref{eqn:aux44} holds.
The only non-trivial cases are when \(x, y \in X\) are distinct and both differ
from \(\theta\). Set \(M \df \eta(\Phi(x),\Phi(y))\). First assume \(M = 0\).
Then \(d(x,\theta) \neq d(y,\theta)\) and hence \(d(x,y) = \max(d(x,\theta),
d(y,\theta))\) or, equivalently, \(d(x,y) = ab^{\max(\mu_0(x),\mu_0(y))} =
D(\Phi(x),\Phi(y))\). Finally, assume \(M > 0\). For simplicity, set \(c \df
c(\mu_0(x),\ldots,\mu_{M-1}(x))\), \(u \df c(\mu_0(x),\ldots,\mu_M(x))\) and \(v
\df c(\mu_0(y),\ldots,\mu_M(y))\), and observe that \(u \neq v\) and \(u, v \in
E(c,\mu_0(x)-M+1)\) and \(x, y \in B(c,\mu_0(x)-M+1))\). So, \(d(x,y) \leq
ab^{\mu_0(x)-M+1}\) and \(d(u,v) = ab^{\mu_0(x)-M+1}\). On the other hand, \(x
\in B(u,\mu_0(x)-M)\) and, similarly, \(y \in B(v,\mu_0(x)-M)\) and therefore
\(d(x,y) = ab^{\mu_0(x)-M+1} = D(\Phi(x),\Phi(y))\) (since \(d\) is
an ultrametric), and we are done.
\end{proof}

To establish the existence of type 1 spaces with arbitrarily preassigned
parameters \(N\), \(a\) and \(b\), we need the next to lemmas.

\begin{lem}{product-1}
Let \((Y,d_Y)\) and \((Z,d_Z)\) be two spaces of type 1 such that
\[(a(Y),b(Y)) = (a(Z),b(Z)).\]
Equip \(W \df Y \times Z\) with the maximum metric \(d_W\) induced by \(d_Y\)
and \(d_Z\); that is, \(d_W((x,z),(y,w)) \df \max(d_Y(x,y),d_Z(z,w))\). Then
\((W,d_W)\) is a (two-point dilation-homogeneous locally compact) space of type
1 as well such that
\[(N(W),a(W),b(W)) = (N(Y) \cdot N(Z),a(Y),b(Y)).\]
\end{lem}
\begin{proof}
Set \((\alpha,\beta) \df (a(Y),b(Y))\). It is straightforward that \(W\) is
an ultrametric space such that \(\Gamma(W) = \{\alpha\beta^k\dd\ k \in \ZZZ\}\)
(this implies that \((a(W),b(W)) = (\alpha,\beta)\)). Further, being the product
of two metrically homogeneous spaces, \(W\) is metrically homogeneous as well.
So, we infer from \THM{ultram-ultrah} that \(W\) is ultrahomogeneous. We also
know that for any \(k \in \ZZZ\) there are dilations \(u_k \in \Lambda(Y)\) and
\(v_k \in \Lambda(Z)\) such that \(|u_k| = |v_k| = \beta^k\). Then it is readily
seen that also \(w_k\dd W \ni (y,z) \mapsto (u_k(y),v_k(z)) \in W\) is
a dilation from \(\Lambda(W)\) such that \(|w_k| = \beta^k\). Now fix three
distinct points \(\xi,\eta,\nu\) from \(W\) and choose \(p \in \ZZZ\) such that
\(d_W(\xi,\eta) = d_W(\xi,\nu) \cdot \beta^p\). Set \(\xi' \df w_p(\xi)\) and
\(\nu' \df w_p(\nu)\) and note that \(d_W(\xi,\eta) = d_W(\xi',\nu')\). So, we
infer from ultrahomogeneity of \(W\) that there exists a bijective isometry
\(q\dd W \to W\) that sends \(\xi\) to \(\xi'\) and \(\eta\) to \(\nu'\). Then
\(r \df w_p^{-1} \circ q \in \Lambda(W)\) satisfies \(r(\xi) = \xi\) and
\(r(\eta) = \nu\), and therefore \(W\) is two-point dilation-homogeneous.
Finally, if \(A \subset Y\) and \(B \subset Z\) are maximal sets of equidistal
points whose distance is \(\alpha\), then \(A \times B\) is a maximal subset of
\(W\) of equidistal points whose distance is \(\alpha\) as well, and therefore
\(N(W) = N(Y) \cdot N(Z)\).
\end{proof}

\begin{lem}{field}
For each finite field \(\FFF_q\) (of size \(q\)) and two reals \(\alpha\) and
\(\beta\) satisfying \(1 \leq \alpha < \beta\), the formula \(\rho(x,y) \df
ab^{-|x-y|}\) correctly defines a metric on \(Z \df \FFF_q((X))\) such that
the metric space \((Z,\rho)\) is two-point dilation-homogeneous and locally
compact, and satisfies \((N(Z),a(Z),b(Z)) = (q,\alpha,\beta)\).
\end{lem}
\begin{proof}
We will only show that \(Z\) is two-point dilation-homogeneous and \(N(Z) = q\)
(all other parts of the result are left to the reader). The former of these
properties is an immediate consequence of the fact that all functions of
the form \(Z \ni x \mapsto px+q \in Z\) are dilations (where \(p, q \in Z\) are
arbitrarily chosen); whereas the latter one is implied by the fact that all
constant monomials (that is, power series of the form \(c X^0\) where \(c \in
\FFF_q\)) form a maximal subset of \(Z\) consisting of equidistal points.
\end{proof}

\begin{rem}{field}
The argument used in the above proof to show that \(Z\) is two-point
dilation-homogeneous works for arbitrary fields \(\KKK\) equipped with
a multiplicative norm (that is, with a function \(|\cdot|\dd \KKK \to
[0,\infty)\) that vanishes only at the origin of \(\KKK\) and satisfies \(|xy| =
|x| \cdot |y|\) and \(|x+y| \leq |x|+|y|\) for all \(x, y \in \KKK\)). In
particular, the fields of \(p\)-adic numbers (and all other local fields) are
two-point dilation-homogeneous locally compact metric spaces.
\end{rem}

The next result summarizes all that have been already established in this
section and covers the case of type 1 spaces of \THM{main}.

\begin{thm}{type-1}
\begin{enumerate}[\upshape(A)]
\item Each space of the form \(F = F_{n,a,b}\) described in item \UP{(Type1)} of
 \THM{main} is a two-point dilation-homogeneous locally compact space of type~1
 such that \((N(F),a(F),b(F)) = (n,a,b)\).
\item Each two-point dilation-homogeneous locally compact space \(X\) of type~1
 is isometric to the space \(F_{n,a,b}\) with \((n,a,b) \df (N(X),a(X),b(X))\).
\end{enumerate}
\end{thm}
\begin{proof}
Just apply \THM{isometric-1} and Lemmas~\ref{lem:product-1} and \ref{lem:field}.
The details are left to the reader.
\end{proof}

\section{Type 2 spaces}

The main tool of this section is the following result due to Tits \cite{tit}.
For its modern proof, see also \cite[Theorem~5.14]{kra}.

\begin{thm}{tits}
Let \(G\) be a transformation group of a locally compact Hausdorff space \(X\)
that acts two-transitively on \(X\), and let \(\theta \in X\). If \(G\) is both
locally compact and \(\sigma\)-compact, and \(X\) is neither compact nor totally
disconnected, then \(G\) is the semidirect product of \(G_{\theta}\) and \(V\)
where:
\begin{enumerate}[\upshape(Pr1)]
\item \(G_{\theta}\) is the isotropy group at \(\theta\); that is,
 \(G_{\theta}\) consists of all \(h \in G\) such that  \(h(\theta) = \theta\);
\item \(V\) is a closed normal subgroup of \(G\), isomorphic to a vector group
 \(\RRR^n\) for some \(n > 0\);
\item \(H\dd V \ni g \mapsto g(\theta) \in X\) is a homeomorphism;
\item after transforming the vector space structure of \(V\) to \(X\) via \(H\),
 \(G_{\theta}\) consists of linear operators on \(X\) and acts transitively on
 \(X \setminus \{\theta\}\).
\end{enumerate}
\end{thm}

We recall that a transformation group \(G\) of a space \(X\) acts
two-transitivily on \(X\) if for any two pairs \((a,b), (p,q) \in X \times X\)
such that \(a \neq b\) and \(p \neq q\) there exists \(g \in G\) satisfying
\(g(a) = p\) and \(g(b) = q\).\par
With the aid of the above result, we will give a full classification of type 2
spaces. To this end, we fix such a space \((X,d)\). It follows from \LEM{H-B}
that \(X\) is non-compact.

\begin{lem}{Lamb-non-tot-disc}
\(\Lambda(X)\) is locally compact, \(\sigma\)-compact and not totally
disconnected.
\end{lem}
\begin{proof}
It follows from Lemmas~\ref{lem:loc-comp}, \ref{lem:H-B}, \ref{lem:group} and
\ref{lem:homo-R} that \(\Lambda(X)\) is locally compact and \(\sigma\)-compact,
and that \(m\dd \Lambda(X) \ni u \mapsto |u| \in (0,\infty)\) is a continuous
surjective homomorphism. Assume, on the contrary, that \(\Lambda(X)\) is totally
disconnected. We conclude from the Open Mapping Theorem (for homomomorphisms
between locally compact \(\sigma\)-compact groups) that \(m\) is open, and from
the classical van Dantzig theorem \cite{v-d} that \(\lambda(X)\) contains
a subgroup \(K\) that is both open and compact. Then \(m(K) = \{1\}\) and \(K\)
has (at most) countable index in \(\Lambda(X)\). Consequently, \(m(K)\) is
countable, which is impossible.
\end{proof}

\begin{cor}{X-not-tot-disc}
The space \(X\) is not totally disconnected.
\end{cor}
\begin{proof}
Let \(S\) stand for the connected component of \(\Lambda(X)\) at the identity.
It follows from \LEM{Lamb-non-tot-disc} that \(S\) is non-trivial. So, there is
\(x \in X\) and \(h \in S\) for which \(h(x) \neq x\). Then the set \(\{g(x)\dd\
g \in S\}\) is connected and contains at least two points.
\end{proof}

The above two results imply that all the assumptions of \THM{tits} are
satisfied. So, after fixing \(\theta \in X\), we may endow \(X\) with
a finite-dimensional vector space structure so that \(\theta\) is the origin of
\(X\), the topology of \((X,d)\) coincides with the vector space topology, all
translations in \(X\) are dilations, and each dilation on \(X\) is of the form
\(A \circ T_a\) where \(A\) belongs to a certain subgroup \(L\) of \(GL(X)\) and
\(T_a\) is the translation by \(a\) (that is, \(T_a(x) = x+a\)). In all
the results that follow we continue these arrangements.

\begin{lem}{invariant}
The metric \(d\) is invariant on the group \(X\).
\end{lem}
\begin{proof}
We only need to show that each translation is an isometry. To this end, first
note that \(A \circ T_x \circ A^{-1} = T_{Ax}\) for any \(A \in L\) and each
\(x \in X\). Now fix non-zero \(e \in X\) and choose \(A \in \Lambda(X)\) that
fixes \(\theta\) and sends \(e\) onto \(-e\). Then necessarily \(A \in L\) and
\(A T_e^{-1} A^{-1} T_e = T_{2e}\). We conclude from \LEM{homo-R} that
\(|T_{2e}| = 1\), and we are done.
\end{proof}

\begin{lem}{R-n-p}
Let \(n\) denote the linear dimension of the vector space \(X\). There exist
an invariant metric \(p\) on \(\RRR^n\) (compatible with the topology of that
space) and an isometry \(Q\dd (X,d) \to (\RRR^n,p)\) such that:
\begin{equation}\label{eqn:sphere}
p(x,y) = p(x',y') \iff d_e(x,x') = d_e(y,y') \qquad (x,x',y,y' \in \RRR^n). 
\end{equation}
\end{lem}
\begin{proof}
Let \(K\) consist of all \(u \in \Lambda(X)\) such that \(u(\theta) = \theta\)
and \(|u| = 1\). Then \(K\) is a compact subgroup of \(L\) (as \(K\) is a closed
subgroup of \(L\) contained in \(U(\theta,a,2\rho(\theta,a))\) where \(a \in X\)
is different from \(\theta\)---consult \LEM{loc-comp}). Since \(L\) is
a subgroup of \(GL(X)\), it follows from Wyel's unitarian trick that there is
a scalar product \(\scalarr\) on \(X\) such that each operator \(A\) from \(K\)
satisifes
\begin{equation}\label{eqn:aux2}
\scalar{Ax}{Ay} = \scalar{x}{y} \qquad (x, y \in X).
\end{equation}
Now let \(Q\dd X \to \RRR^n\) be a linear isomorphism such that \(\scalar{x}{x}
= \|Qx\|_2^2\) for all \(x \in X\). We define a compatible metric \(p\) on
\(\RRR^n\) by \(p(x,y) \df d(Q^{-1}(x),Q^{-1}(y))\). Since \(Q\) is an isometry
from \((X,d)\) onto \((\RRR^n,p)\) that is linear, we infer that \(p\) is
invariant (thanks to \LEM{invariant}). Moreover, if \(p(x',y') = p(x,y)\), then
\(d(Q^{-1}(x')-Q^{-1}(y'),\theta) = d(Q^{-1}(x)-Q^{-1}(y),\theta)\) and it
follows from two-point dilation-ho\-mogeneity of \(X\) that there is \(A \in K\)
that sends \(Q^{-1}(x)-Q^{-1}(y)\) to \(Q^{-1}(x')-Q^{-1}(y')\). We conclude
from \eqref{eqn:aux2} that \(Q A Q^{-1}\) is isometric w.r.t. \(d_e\) and hence
\(d_e(x'-y',0) = d_e((QAQ^{-1})x-(QAQ^{-1})y,0) = d_e(x-y,0)\) or, equivalently,
\(d_e(x',y') = d_e(x,y)\). This shows the direct implication of
\eqref{eqn:sphere}. To prove the reverse one, one uses a general property of
connected metric spaces that each sphere in such a space disconnects this space.
More precisely, if \(n > 1\), then the \(p\)-sphere \(\{x \in \RRR^n\dd\ p(0,x)
= r\}\) disconnects \(\RRR^n\) and is contained in the standard sphere \(\{x \in
\RRR^n\dd\ \|x\|_2 = r\}\), which implies that both these spheres coincide. And
if \(n = 1\), then the \(p\)-sphere \(S\) around \(0\) is non-empty and
symmetric (that is, \(S = -S\)), and contained in the standard two-point sphere,
which again implies that these two spheres coincide. So, the reverse implication
of \eqref{eqn:sphere} follows from the invariance of both the metrics involved
therein (and the identities of the respective spheres around \(0\)).
\end{proof}

\begin{lem}{monotone}
Let \(\omega\dd [0,\infty) \to [0,\infty)\) be a one-to-one function.
\begin{enumerate}[\upshape(A)]
\item If \(\omega \circ d_e\) is a metric, then \(\omega(x+y) \leq \omega(x)+
 \omega(y)\) for all \(x, y \geq 0\).
\item If \(\omega\) is continuous and surjective, and both \(\omega \circ d_e\)
 and \(\omega^{-1} \circ d_e\) are metrics, then \(\omega(x) = c x\) for all \(x
 \geq 0\) and some constant \(c > 0\).
\end{enumerate}
\end{lem}
\begin{proof}
Assume \(\omega \circ d_e\) is a metric and fix \(x, y \in [0,\infty)\). Take
any vector \(v \in \RRR^n\) such that \(d_e(v,0) = 1\). Then \(\omega(x+y) =
\omega(d_e((x+y)v,0)) \leq \omega(d_e((x+y)v,yv)) + \omega(d_e(yv,0)) =
\omega(x)+\omega(y)\).\par
Now if both \(\omega \circ d_e\) and \(\omega^{-1} \circ d_e\) are metrics, then
for \(a \df \omega(x)\) and \(b \df \omega(y)\) we have (from (A))
\(\omega^{-1}(a+b) \leq \omega^{-1}(a)+\omega^{-1}(b)\) or, equivalently,
\(\omega(x)+\omega(y) \leq \omega(x+y)\). So, \(\omega(x+y) = \omega(x)+
\omega(y)\) for any \(x, y \geq 0\) and the conclusion of (B) follows from
continuity and monotonicity of \(\omega\).
\end{proof}

Finally, we are able to give a proof of the following

\begin{thm}{type-2}
If \((X,d)\) is a two-point dilation-homogeneous metric space that contains
a compact set with non-empty interior and the set \(d(X \times X)\) is dense in
\([0,\infty)\), then there exists a unique integer \(n > 0\) and real \(\alpha
\in (0,1]\) such that \((X,d)\) is isometric to \((\RRR^n,d_e^{\alpha})\).
\end{thm}
\begin{proof}
We have already shown that if \(X\) has all the postulated properties, then it
is Heine-Borel and type 2, and it follows from \LEM{R-n-p} that \((X,d)\) is
isometric to \((\RRR^n,p)\) for some \(n > 0\) and a compatible invariant metric
\(p\) on \(\RRR^n\) that satisfies \eqref{eqn:sphere}. In particular, for any
\(r > 0\) there is a dilation \(u_r \in \Lambda(\RRR^n,p)\) with \(|u_r| =
r\).\par
It follows from \eqref{eqn:sphere} that there exists a one-to-one function
\(\omega\dd [0,\infty) \to [0,\infty)\) such that \(p = \omega \circ d_e\).
Since \(p\) is compatible, we infer that \(\omega\) is continuous at \(0\). Now
item (A) of \LEM{monotone} yields that \(\omega\) is a continuous function. In
particular, it is strictly increasing. And \LEM{H-B} implies that
\(\omega([0,\infty)) = [0,\infty)\).\par
Now fix arbitrary \(r > 0\) and set \(\tau\dd [0,\infty) \ni x \mapsto
\omega^{-1}(r \omega(x)) \in [0,\infty)\). Observe that \(\tau\) is a strictly
increasing (and continuous) bijection on \([0,\infty)\), and for all \(x, y \in
\RRR^n\):
\begin{itemize}
\item \((\tau \circ d_e)(x,y) = d_e(u_r(x),u_r(y))\); and
\item \((\tau^{-1} \circ d_e)(x,y) = d_e(u_{r^{-1}}x,u_{r^{-1}}y)\).
\end{itemize}
In particular, both \(\tau \circ d_e\) and \(\tau^{-1} \circ d_e\) are metrics.
So, \LEM{monotone} yields that \(\tau(x) = \phi(r) x\) for some constant
\(\phi(r) > 0\). In this way we have obtained a function \(\phi\dd (0,\infty)
\to (0,\infty)\) such that
\begin{equation}\label{eqn:aux3}
r \omega(x) = \omega(\phi(r)x) \qquad (r > 0,\ x \geq 0).
\end{equation}
Since \(\phi(r) = \omega^{-1}(r \omega(1))\), we see that \(\phi\) is continuous
and strictly increasing. Moreover, it easily follows from \eqref{eqn:aux3} that
\(\phi\) is an endomorphism of the multiplicative group of all positive reals.
So, we conclude that \(\phi(r) = r^{\beta}\) for some real constant \(\beta >
0\). Setting \(\alpha \df \frac{1}{\beta}\) and \(c \df \omega(1)^{\beta}\), and
substituting \(x = 1\) and \(r = t^{\alpha}\) in \eqref{eqn:aux3}, we obtain
\(\omega(t) = (ct)^{\alpha}\). Finally, since \(\omega\) satisfies condition (A)
of \LEM{monotone}, we get \(\alpha \leq 1\). So, \(p = (c d_e)^{\alpha}\). Since
\((\RRR^n,d_e)\) is isometric to \((\RRR^n,c d_e)\), we infer that \((X,d)\) is
isometric to \((\RRR^n,d_e^{\alpha})\). Uniqueness of \(n\) and \(\alpha\) are
left to the reader (cf. \COR{get-param} below).
\end{proof}

Below we keep the notation of \THM{main}. The proof of the next result is left
to the reader.

\begin{cor}{get-param}
Let \((X,d)\) be a two-point dilation-homogeneous metric space that has more
than one point and contains a compact set with non-empty interior. Further, let
\(r > 0\) be an arbitrary distance attainable in \(X\) and let \(A\) be
a maximal subset of \(X\) whose all points are at a distance of \(r\).
\begin{enumerate}[\upshape(A)]
\item If \(d(X \times X) = \{0,r\}\), then \((X,d)\) is isometric to
 \(D_{\alpha,r}\) where \(\alpha\) is the cardinality of \(A\).
\item If \(d(X \times X)\) contains more than two numbers and is not dense in
 \([0,\infty)\), then let \(n\), \(a\), and \(b'\) denote, respectively,
 the size of \(A\), the least distance not less than \(1\) attainable in \(X\)
 and the least distance attainable in \(X\) that is greater than \(a\). Then
 \(n\) is finite, greater than \(1\) and \((X,d)\) is isometric to \(F_{n,a,b}\)
 with \(b = \frac{b'}{a}\).
\item If \(d(X \times X)\) is dense in \([0,\infty)\), then \((X,d)\) is
 isometric to \(R_{n-1,\alpha}\) where \(n > 1\) is the (finite) size of \(A\)
 and \(\alpha \df \log_2\frac{\diam \bar{B}(a,r)}{r}\ (\in (0,1])\) where \(a
 \in X\) is chosen arbitrarily and \(\diam B\) stands for the diameter of
 a subset \(B\) of \(X\).
\end{enumerate}
\end{cor}

An attentive reader will realise that Theorems~\ref{thm:main} and \ref{thm:R-n}
are direct consequences of the above result. For the reader's convenience, below
we give a brief

\begin{proof}[Proof of \COR{abs-homo}]
Assume \(X\) has all the properties postulated in the result. If \(X\) is of
type 2, then the conclusion easily follows from an analogous property of
the Euclidean spaces. And if \(X\) is of type 0, the assertion is clear (as then
\(X\) is finite). So, we only need to discuss the case when \(X\) is of type 1.
To this end, consider an arbitrary locally compact two-point
dilation-homogeneous metric space \((X,d)\) and a dilation \(\psi\dd A \to B\)
between two arbitrary subsets of \(X\). If \(A\) has at most one element, then
\(\psi\) extends to an isometry, by \LEM{two-point}. So, below we assume \(A\)
has at least two points, say \(a, b \in A\) with \(a \neq b\). Then there exists
\(\xi \in \Lambda(X)\) that sends \(a\) to \(\psi(a)\) and \(b\) to \(\psi(b)\).
In particular, \(\phi\dd A \ni x \mapsto \xi^{-1}(\psi(x)) \in X\) is isometric,
and it suffices to show that \(\phi\) extends to an isometry \(\Phi\dd X \to X\)
(because then \(\Psi \df \xi \circ \Phi\) is a dilation that extends
\(\psi\)).\par
We conclude from \COR{ultra} that \(X\) is ultrahomogenous. Observe that each
function \(u\dd X \to X\) isometric on a set \(F \subset X\) containing \(a\)
and satisfying \(u(a) = \phi(a)\), for each \(x \in F\) satisfies \(u(x) \in K_x
\df \bar{B}(\phi(a),d(a,x))\) (\(K_a \df \{\phi(a)\})\); and note that \(K_x\)
is a compact set for each \(x \in X\). Now denote by \(I\) the collection of all
finite subsets of \(X\) that contain \(a\) and for any \(F \in I\) denote by
\(\LlL(F)\) the collection of all functions \(u\dd X \to X\) such that:
\begin{itemize}
\item \(u(x) = \phi(x)\) for any \(x \in F \cap A\) (in particular, \(u(a) =
 \phi(a)\)); and
\item \(u\) is isometric on \(F\); and
\item \(u(x) \in K_x\) for any \(x \in X\).
\end{itemize}
Since \(X\) is ultrahomogeneous, \(\LlL(F) \neq \varempty\) for each \(F \in
I\). Moreover, it is readily seen that \(\LlL(F_1) \cap \LlL(F_2) \supset
\LlL(F_1 \cup F_2)\) for any \(F_1, F_2 \in I\). So, the family
\(\{\LlL(F)\}_{F \in I}\) is centered. On the other hand, each of \(\LlL(F)\)
may naturally be considered as a subspace of \(Z \df \prod_{x \in X} K_x\) and
as such, it is a closed subspace of \(Z\). So, the Tychonoff theorem implies
that \(\bigcap_{F \in I} \LlL(F)\) is non-empty. An observation that each
function \(\Phi\) from this intersection is an isometry that extends \(\phi\)
finishes the proof. (\(\Phi\) needs to be surjective, because otherwise we may
apply the whole argument to \(\phi = \Phi^{-1}\) to get a contradiction.)
\end{proof}

We end the paper with the following

\begin{exm}{prod}
As we have seen, locally compact two-point dilation-homogeneous metric spaces
are uniquely determined by two sets of characteristic parameters, one of which
is responsible for the `size' of the underlying space, whereas the other for
the `shape' of the metric:
\begin{itemize}
\item For type 0 spaces \(D_{\alpha,r}\): size parameter is \(\alpha\);
 metric parameter is \(r\).
\item For type 1 spaces \(F_{n,a,b}\): size parameter is \(n\); metric
 parameters are \(a\) and \(b\).
\item For type 2 spaces \(R_{n,\alpha}\): size parameter is \(n\); metric
 one is \(\alpha\).
\end{itemize}
It is a strange phenomenon that when dealing with spaces having the same type
and a common set of metric parameters, we may produce their products in a quite
similar manner, but the size parameter changes differently:
\begin{itemize}
\item For a (non-empty) collection \(\{(X_s,d_s)\}_{s \in S}\) (finite if
 dealing with type 1) of spaces of common type different than 2 and with common
 metric parameters, its product \(Y \df \prod_{s \in S} X_s\) is equipped with
 the `sup' metric (that is, \(d_Y((x_s)_{s \in S},(y_s)_{s \in S}) =
 \sup_{s \in S} d_s(x_s,y_s)\)); and the size parameter of \(Y\) coincides with
 the product of size parameters of all spaces \(X_s\).
\item For a finite collection \((X_1,d_1),\ldots,(X_N,d_N)\) of type 2 spaces
 with common metric parameter \(\alpha\), their product \(Y = X_1 \times \ldots
 \times X_N\) is equipped with the \(\ell_p\)-metric where \(p =
 \frac{1}{\alpha}\) (that is, \(d_Y((x_1,\ldots,x_N),(y_1,\ldots,y_N)) =
 (\sum_{k=1}^N d_k(x_k,y_k)^p)^{\alpha}\)); and the size parameter of \(Y\)
 coincides with the sum of size parameters of all spaces \(X_k\), reduced by
 \(N-1\) (that is, if \(s(X)\) stands for the size parameter of a space \(X\),
 then \(s(X_1 \times \ldots \times X_N) = (\sum_{k=1}^N s(X_k)) + 1 - N\)).
\end{itemize}
We leave the details to interested readers.
\end{exm}

\end{document}